\documentclass[preprint,11pt]{elsarticle}
\usepackage{amsmath,amssymb}
\usepackage[all]{xy}
\usepackage{latexsym}
\usepackage{amsthm,a4wide,color}
\usepackage{amsmath,amscd,verbatim}
\usepackage{hyperref}
\input diagxy

\theoremstyle{plain}
  \newtheorem{thm}{Theorem}[section]
  \newtheorem{lem}[thm]{Lemma}
  \newtheorem{prop}[thm]{Proposition}
  \newtheorem{cor}[thm]{Corollary}
\theoremstyle{definition}
  \newtheorem{defn}[thm]{Definition}
  
  \newtheorem{exmp}[thm]{Example}

\makeatletter \def\ps@pprintTitle{  \let\@oddhead\@empty  \let\@evenhead\@empty  \def\@oddfoot{\centerline{\thepage}}  \let\@evenfoot\@oddfoot} \makeatother

\begin{document}

\newcommand{\oto}{{\to\hspace*{-3.1ex}{\circ}\hspace*{1.9ex}}}
\newcommand{\lam}{\lambda}
\newcommand{\da}{\downarrow}
\newcommand{\Da}{\Downarrow\!}
\newcommand{\D}{\Delta}
\newcommand{\ua}{\uparrow}
\newcommand{\ra}{\rightarrow}
\newcommand{\la}{\leftarrow}
\newcommand{\Lra}{\Longrightarrow}
\newcommand{\Lla}{\Longleftarrow}
\newcommand{\rat}{\!\rightarrowtail\!}
\newcommand{\up}{\upsilon}
\newcommand{\Up}{\Upsilon}
\newcommand{\ep}{\epsilon}
\newcommand{\ga}{\gamma}
\newcommand{\Ga}{\Gamma}
\newcommand{\Lam}{\Lambda}
\newcommand{\CF}{{\cal F}}
\newcommand{\CG}{{\cal G}}
\newcommand{\CH}{{\cal H}}
\newcommand{\CI}{{\cal I}}
\newcommand{\CB}{{\cal B}}
\newcommand{\CT}{{\cal T}}
\newcommand{\CS}{{\cal S}}
\newcommand{\CV}{{\cal V}}
\newcommand{\CP}{{\cal P}}
\newcommand{\CQ}{{\cal Q}}
\newcommand{\mq}{\mathcal{Q}}
\newcommand{\cu}{{\underline{\cup}}}
\newcommand{\ca}{{\underline{\cap}}}
\newcommand{\nb}{{\rm int}}
\newcommand{\Si}{\Sigma}
\newcommand{\si}{\sigma}
\newcommand{\Om}{\Omega}
\newcommand{\bm}{\bibitem}
\newcommand{\bv}{\bigvee}
\newcommand{\bw}{\bigwedge}
\newcommand{\lra}{\longrightarrow}
\newcommand{\tl}{\triangleleft}
\newcommand{\tr}{\triangleright}
\newcommand{\dda}{\downdownarrows}
\newcommand{\dia}{\diamondsuit}
\newcommand{\y}{{\bf y}}
\newcommand{\colim}{{\rm colim}}
\newcommand{\fR}{R^{\!\forall}}
\newcommand{\eR}{R_{\!\exists}}
\newcommand{\dR}{R^{\!\da}}
\newcommand{\uR}{R_{\!\ua}}
\newcommand{\swa}{{\swarrow}}
\newcommand{\sea}{{\searrow}}
\newcommand{\bbA}{{\mathbb{A}}}
\newcommand{\bbB}{{\mathbb{B}}}
\newcommand{\bbC}{{\mathbb{C}}}
\numberwithin{equation}{section}
\renewcommand{\theequation}{\thesection.\arabic{equation}}

\begin{frontmatter}
\title{Sober  metric approach spaces}%\tnoteref{F}} \tnotetext[F]{This work is supported by National Natural Science Foundation of China (11371265).}

\author{Wei Li}
\ead{mathli@foxmail.com}
\author{Dexue Zhang}%\corref{cor}}
\ead{dxzhang@scu.edu.cn}
 %\cortext[cor]{Corresponding author.}
\address{School of Mathematics, Sichuan University, Chengdu 610064, China}

\begin{abstract}It is proved that a metric space is sober,  as an approach space, if and only if it is Smyth complete. \end{abstract} \begin{keyword}Metric space \sep Yoneda completeness \sep Smyth completeness \sep approach space \sep metric  approach space \sep sober approach space \MSC[2010] 18B30 \sep 18B35 \sep   54B30
\sep 54E99	 \end{keyword}
 \end{frontmatter}

%\tableofcontents

\section{Introduction}

Approach spaces,   introduced by Lowen  \cite{RL89}, are a common extension of topological spaces and metric spaces. By a metric on a set $X$ we understand, as in Lawvere \cite{Lawvere73}, a map $d:X\times X\rightarrow [0,\infty]$ such that $d(x,x)=0$ and $d(x,y)+d(y,z)\geq d(x,z)$  for all $x,y, z\in X$.
An extensive investigation of approach spaces can be found in the monographs of Lowen  \cite{RL97,Lowen15}. An approach space is said to be a topological one if it is generated by a topological space; and it is said to be a  metric one if it is generated by a metric space.

Sober approach spaces, a counterpart of sober topological spaces in the metric setting, are introduced in \cite{BRC}. It is proved there that a topological space is sober as an approach space, if and only if it is sober  as a topological space. So, it is natural   to ask what kind of  metric approach spaces are sober? A partial answer is obtained in \cite{BRC}. If $d$ is a usual metric (i.e., a   symmetric, separated  and finitary metric) on a set $X$,  it follows from Corollary 5.19 in \cite{BRC} that $(X,d)$ is  sober, as an approach space, if and only if $(X,d)$ is a complete metric space.
This paper presents a complete answer to this question. The answer is a bit surprising:  a metric space   is sober, as an approach space, if and only if it is Smyth complete.
A metric space is Smyth complete if every forward Cauchy net in it converges in its symmetrization \cite{Goubault,KS2002}.
Smyth completeness originated in the works of Smyth \cite{Smyth87,Smyth94} that aimed to provide a common framework for the domain approach  and the metric space approach to semantics in computer science.

As advocated in \cite{GH,Hof2011,HST}, in this paper  we emphasize that the relationship between  approach spaces and  metric spaces is analogous to that between topological spaces and ordered sets. This point of view has proved to be fruitful, and  is well in accordance with the thesis of Smyth \cite{Smyth87} ``that domains are, or should be, a prime area for the application of quasi-uniform ideas, and can help us to get the definitions right."

An order on a set $X$ is a map $X\times X\rightarrow  \{0,1\}$ fulfilling certain requirements; a topology (identified with the corresponding closure operator) is a map $X\times 2^X\rightarrow\{0,1\}$ (the transpose of the closure operator) that satisfies certain conditions. Replacing the quantale $2=(\{0,1\},\wedge)$  by Lawvere's quantale $([0,\infty]^{\rm op},+)$ in the postulations of ordered sets and topological spaces,   we obtain  metric spaces and approach spaces.

The following commutative squares exhibit some basic relationship among the categories of ordered sets, topological spaces, metric spaces and approach spaces:
\[\bfig \square<700,500>[{\sf Ord}` {\sf Top}` {\sf Met}  `{\sf App}; \Gamma` \omega ` \omega ` \Gamma] \square(1500,0)/>`<-`<-`>/<700, 500>[{\sf Top} `{\sf Ord} `{\sf App}`{\sf Met} ;  \Omega `\iota`\iota`\Omega] \efig\] where, \begin{itemize}  \item  the involved categories are ``self evident", and will be explained in the next section; \item the top row:   $\Gamma$ sends each  ordered set $(X,\leq)$ to its Alexandroff topology,   $\Omega $ sends a topological space to its specialization order; \item the bottom row: $\Gamma$ sends a metric space to the corresponding metric approach space, $\Omega$  sends an approach space to its specialization metric; \item $\omega$ (in both cases) is a full and faithful functor with  a right adjoint given by $\iota$. \end{itemize}

These facts can be found in \cite{RL97}. The bottom row is an analogy  of the  top row in the metric setting. In particular, approach spaces extend metric spaces, via the functor $\Gamma$, in the same way as topological spaces extend ordered sets. The problem considered in this paper is to characterize those metric spaces $(X,d)$ for which $\Gamma(X,d)$ are sober. To this end, some properties of the other functors  will also be considered. The main results include:

(1) The specialization metric of a sober approach space is Yoneda complete (Proposition \ref{Sober implies Yoneda}). This is an analogy  in the metric setting of the fact that the specialization order of a sober topological space is directed complete.

(2) For a metric space $(X,d)$, the specialization metric space of the sobrification of $\Gamma(X,d)$ coincides with the Yoneda completion of $(X,d)$    (Theorem \ref{5.8}).

(3) For a metric space $(X,d)$, the approach space $\Gamma(X,d)$ is sober if and only if $(X,d)$ is Smyth complete (Theorem \ref{main}).

\section{Topological spaces, metric  spaces, and approach spaces}
Write $2$ for the quantale (i.e., a small and complete monoidal closed category) $(\{0,1\},\wedge)$. An ordered set is then a $2$-enriched category. Precisely,  an ordered set is a set $X$ together with a map $p:X\times X\rightarrow 2$ such that for all $x,y,z\in X$: \begin{enumerate}
\item[\rm (P1)]   $p(x,x)=1$,   \item[\rm (P2)]  $p(x,y)\wedge p(y,z)\leq p(x,z)$.  \end{enumerate}

It is  traditional to write  $x\leq y$ for $p(x,y)=1$ in order theory.

Given a topological space $X$, the closure operator on $X$ induces a map $c:X\times 2^X\rightarrow 2$, given by
\[c(x,A)=\left\{\begin{array}{ll}1, & x\in \overline{A},\\
0, &x\notin \overline{A}.\end{array}\right.\] This map satisfies the following conditions: \begin{enumerate}
\item[\rm (C1)] $c(x, \{x\})=1$,
\item[\rm (C2)] $c(x, \emptyset)=0$,
\item[\rm (C3)] $c(x, A\cup B)=c(x, A)\vee c(x, B)$,
\item[\rm (C4)]
$c(x,A)\geq c(x,B)\wedge\bw_{y\in B}c(y,A)$. \end{enumerate}

The condition (C4) expresses the idempotency of the closure operator. Topologies on a set $X$ correspond  bijectively to maps $c:X\times 2^X\rightarrow 2$ that satisfy the conditions (C1)-(C4).

The \emph{specialization order} \cite{Johnstone} of a topological space $X$ is the composite \[X\times X\to^{(x,y)\mapsto(x, \{y\})}X\times 2^X\to^c 2;\]  or equivalently,    $x\leq y$ if $x\in\overline{\{y\}}$.   Taking specialization order defines  a  functor \[\Omega:{\sf Top}\rightarrow{\sf Ord}\] from the category of topological spaces and continuous maps to the category ${\sf Ord}$  of  ordered sets and order-preserving maps. The functor $\Omega$ has a left adjoint \[\Gamma:{\sf Ord}\rightarrow{\sf Top}\] that maps an ordered set $(X,\leq)$ to the space obtained by endowing $X$ with the Alexandroff topology of $(X,\leq)$ (i.e., the topology whose closed sets are the lower subsets in $(X,\leq)$).

A non-empty closed subset $A$ of a topological space $X$ is irreducible if for any closed subsets $B,C$, $A\subseteq B\cup C$  implies $A\subseteq B$ or $A\subseteq C$.
A topological space $X$ is sober if for each irreducible closed subset $A$, there exists a unique $x\in X$ such that $A$ equals the  closure of $\{x\}$. It is well-known that the specialization order of a sober  topological  space is directed complete, i.e., every directed set in it has a join \cite{GH,Johnstone}.

\begin{defn} (\cite{Lawvere73}) A metric space is a category enriched over the Lawvere quantale $([0,\infty]^{\rm op},+)$. Explicitly, a metric space   $(X,d)$ consists of a set $X$ and a map $d:X\times X\rightarrow [0,\infty]$ such that $d(x,x)=0$ and $d(x,y)+d(y,z)\geq d(x,z)$ for all $x,y, z\in X$.
The map $d$ is called a metric, and the value $d(x,y)$  the distance from $x$ to $y$. \end{defn}

A metric space $(X,d)$ is \emph{symmetric} if $d(x,y)=d(y,x)$ for all $x,y\in X$;  \emph{separated}  if $x=y$ whenever $d(x,y)=d(y,x)=0$; \emph{finitary} if $d(x,y)<\infty$ for all $x,y\in X$.  A metric space in the usual sense is exactly a symmetric, separated and finitary one. Given a metric $d$ on a set $X$, the opposite $d^{\rm op}$ of $d$  refers to the metric given by $d^{\rm op}(x,y)=d(y,x)$; the symmetrization $d^{\rm sym}$ of $d$ is given by $d^{\rm sym}(x,y)=\max\{d(x,y),d(y,x)\}$.

A non-expansive map $f: (X,d)\rightarrow (Y,p)$ between metric spaces is a map $f:X\rightarrow Y$ such that $d(x,y)\geq p(f(x),f(y))$ for all $x, y$ in $X$. Metric spaces and non-expansive maps form a category, denoted by ${\sf Met}$. A map $f: (X,d)\rightarrow (Y,p)$ between metric spaces is isometric if $d(x,y)=p(f(x),f(y))$ for all $x,y\in X$.

\begin{exmp}[The Lawvere metric, \cite{Lawvere73}]  For any $a,b$ in $[0,\infty]$, the Lawvere distance, $d_L(a,b)$, from $a$ to $b$ is defined to be the truncated minus $b\ominus a$, i.e.,
\[d_L(a,b)=b\ominus a= \max\{0,b-a\},\]
where we take by convention that $\infty-\infty=0$ and $\infty-a=\infty$ for all $a<\infty$. It is clear that $([0,\infty],d_L)$ is a separated, non-symmetric, and non-finitary  metric space.

The opposite of the Lawvere metric is denoted by $d_R$, i.e., $d_R(x,y)=x\ominus y$.
\end{exmp}

 Let $(X,d)$ be a metric space. A weight, a.k.a. a left module \cite{Lawvere73,SV2005}, of $(X,d)$ is a function $\phi:X\rightarrow[0,\infty]$ such that $\phi(x)\leq \phi(y)+d(x,y)$ for all $x,y\in X$. A coweight, a.k.a. a right module, of $(X,d)$ is a function $\psi:X\rightarrow[0,\infty]$ such that $\psi(y)\leq \psi(x)+d(x,y)$ for all $x,y\in X$.
Said differently, a weight of $(X,d)$ is a non-expansive map $\phi:(X, d)\rightarrow([0,\infty],d_R)$; a coweight of $(X,d)$ is a non-expansive map $\psi:(X, d)\rightarrow([0,\infty],d_L)$.

Given a metric space $(X,d)$, let $\CP X$ be the set of all weights of $(X,d)$. It is obvious that $\CP X$ has the following properties:
\begin{enumerate}[(W1)]
\item For each $x\in X$, $d(-,x)\in\CP X$. Such  weights are said to   be  representable.
\item For each subset $\{\phi_i\}_{i\in I}$ of $\CP X$,  both  $\inf_{i\in I}\phi_i$ and $\sup_{i\in I}\phi_i$ are  in $\mathcal{P}X$.
\item For all $\phi\in\mathcal{P}X$ and $\alpha\in[0,\infty]$, both $\phi+\alpha$ and $\phi\ominus\alpha$ are in $\mathcal{P}X$.
\end{enumerate}

For all $\phi,\psi\in\CP X$, let \[\overline{d}(\phi,\psi)=\sup_{x\in X}d_L(\phi(x),\psi(x))\] Then $\overline{d}$ is a separated metric on $\CP X$. For all $x\in X$ and $\phi\in\CP X$, it holds that \[\overline{d}(d(-,x),\phi)=\phi(x).\] In particular, the correspondence  $x\mapsto d(-,x)$ defines an isometric map $(X,d)\rightarrow(\CP X,\overline{d})$. That is, $d(x,y)=\overline{d}(d(-,x),d(-,y))$ for all $x,y\in X$. These facts are instances of the Yoneda lemma and the Yoneda embedding in enriched category theory, see e.g. \cite{Lawvere73}.

%Let  $f:(X,d)\rightarrow(Y,p)$ and $g:(Y,p)\rightarrow(X,d)$ be non-expansive maps between   metric spaces. We say that $f$ is   left adjoint to $g$ (or, $g$ is  right adjoint to $f$), if $p(f(x),y)= d(x,g(y))$ for all $x\in X$ and $y\in Y$. Left and right adjoint non-expansive maps are a special case of left and right  adjoint functors between enriched categories \cite{Kelly}.

\begin{defn}\label{AP} (\cite{RL89,RL97}) An approach space $(X,\delta)$ consists of a set $X$ and  a map $\delta:X\times 2^X \rightarrow [0,\infty]$,    subject to the following conditions:
\begin{enumerate}
\item[\rm (A1)] $\delta(x, \{x\})=0$,
\item[\rm (A2)] $\delta(x, \emptyset)=\infty$,
\item[\rm (A3)] $\delta(x, A\cup B)=\min\{\delta(x, A),\delta(x, B)\}$,
\item[\rm (A4)]  $\delta(x,A)\leq\delta(x,B)+\sup_{b\in B}\delta(b,A)$, \end{enumerate}  for all $x\in X$  and $ A,B\in 2^X$. The map $\delta$ is called an approach distance on $X$.
\end{defn}

It should be noted that in  \cite{RL89,RL97}, instead of (A4), the following condition is used in the definition of approach spaces: \begin{enumerate}
\item[\rm (A4')]  For all $\varepsilon\in[0,\infty]$, $\delta(x, A)\leq \delta(x, A^\varepsilon)+ \varepsilon$, where $A^\varepsilon=\{x\in X|\ \delta(x,A) \leq\varepsilon\}$.
\end{enumerate} In the presence of   (A1)-(A3),   (A4') is equivalent to (A4). The implication (A4'$)\Rightarrow($A4) is   contained in \cite{RL97}. Putting $B=A^\varepsilon$ in (A4) gives the converse implication.

The conditions (A1)-(A4) are  metric version   of (C1)-(C4), respectively. Thus, it can be said that while metric spaces are $[0,\infty]$-valued ordered sets, approach spaces are $[0,\infty]$-valued topological spaces. The theory of approach spaces has been extended to the quantale-valued setting in the recent paper \cite{LT16}.

A contraction $f: (X,\delta)\rightarrow (Y,\rho)$ between approach spaces is a map $f: X\rightarrow Y$ such that $\delta(x,A)\geq \rho(f(x),f(A))$ for all $A\subseteq X$ and $x\in X$.  Approach spaces and contractions form a category, denoted by ${\sf App}$.

Given an approach space $(X,\delta)$,  define $\Omega(\delta) :X\times X\rightarrow[0,\infty]$ by  $\Omega(\delta)(x,y)= \delta(x,\{y\})$, then $\Omega(\delta)$ is a metric on $X$, called the specialization metric of $(X,\delta)$.  The term \emph{specialization metric} is chosen because of its analogy to the specialization order of topological spaces.
The correspondence $(X,\delta)\mapsto (X,\Omega(\delta))$ defines a functor  \[\Omega:{\sf App}\rightarrow{\sf Met}. \] This functor is  a   counterpart of $\Omega:{\sf Top}\rightarrow{\sf Ord}$ in the metric setting. We denote both of them by $\Omega$, since it is easy to detect from the context which one is meant.

Given a metric space $(X,d)$, define $\Gamma(d):X\times 2^X \rightarrow [0,\infty]$ by \[\Gamma(d)(x,A)=\left\{\begin{array}{ll}\infty, &A=\emptyset,\\ \inf\limits_{y\in A}d(x,y),& A\not=\emptyset. \end{array}\right.\]   Then $\Gamma(d)$ is an approach distance on $X$, called the Alexandroff distance generated by $d$.
The correspondence $(X,d)\mapsto(X,\Gamma(d))$ defines a full and faithful functor $$\Gamma:{\sf Met}\rightarrow {\sf App}$$
 that is left adjoint to $\Omega:{\sf App}\rightarrow{\sf Met}$ \cite{RL97}. In particular,  {\sf Met} is a coreflective full subcategory of {\sf App}. A  space of the form $\Gamma(X, d)$ is said to be a metric approach space.  The functor   $\Gamma:{\sf Met}\rightarrow {\sf App}$ is a  metric version of $\Gamma:{\sf Ord}\rightarrow {\sf Top}$.

\begin{exmp}\label{real approach}   For all $x\in[0,\infty]$ and $A\subseteq[0,\infty]$, let
$$\delta_\mathbb{P}(x,A)= \left\{\begin{array}{ll}\max\{x-\sup A,0\}, & A\neq\emptyset,\\
\infty, &A=\emptyset.
\end{array}\right.$$ Then $\delta_\mathbb{P}$ is an approach  distance on $[0,\infty]$. The approach space  $\mathbb{P}=([0,\infty],\delta_\mathbb{P})$ is introduced in Lowen  \cite{RL97}, it plays an important role in the theory of approach spaces.

The specialization metric of $\mathbb{P}$ is the opposite $d_R$ of the Lawvere distance on $[0,\infty]$, i.e., $d_R(a,b)=a\ominus b$. The approach space $\mathbb{P}$ is not a metric one. In fact,   for all $x\in[0,\infty]$ and $A\subseteq[0,\infty]$,
\[\Gamma(d_R)(x,A)= \left\{\begin{array}{ll}\infty, &A=\emptyset, \\ \max\{x-\sup A,0\}, & x\not=\infty, A\neq\emptyset,\\
\infty, & x=\infty, \infty\notin A,\\ 0, & x=\infty\in A.
\end{array}\right.\]
So,  $\delta_\mathbb{P}$ and $\Gamma(d_R)$ are different  approach distances.
\end{exmp}

Approach spaces can be equivalently described in many ways \cite{RL97}, one of them we need  is the description by regular functions.  A  regular function of an approach space $(X,\delta)$ is a contraction $\phi:(X,\delta)\rightarrow\mathbb{P}$, where $\mathbb{P}$ is the  approach space given in Example \ref{real approach}. Explicitly, a regular function of $(X,\delta)$ is a function $\phi:X\rightarrow[0,\infty]$ such that \[\delta(x,A)\geq \phi(x)\ominus \sup \phi(A)\] for all $x\in X$ and all $A\subseteq X$.

For each subset $A$ of $X$, the condition (A4) in the definition of approach spaces ensures that $\delta(-,A)$ is a regular function  of $(X,\delta)$.

The following proposition says that an approach space is uniquely determined by its regular functions.
\begin{prop} (\cite{RL97}) \label{regular functions} Let $(X,\delta)$ be an approach space. Then the set $\mathcal{R}X$ of regular functions of $(X,\delta)$   satisfies the following conditions:
\begin{enumerate}[(R1)]
 \item For each subset $\{\phi_i\}_{i\in I}$ of $\mathcal{R}X$,  $\sup_{i\in I}\phi_i\in \mathcal{R}X$.
\item For all   $\phi,\psi\in\mathcal{R}X$,  $\min\{\phi,\psi\}\in \mathcal{R}X$.
\item For all $\phi\in\mathcal{R}X$ and $\alpha\in[0,\infty]$, both $\phi+\alpha$ and $\phi\ominus \alpha$ are in $\mathcal{R}X$.
\end{enumerate}

Conversely, suppose that $\mathcal{S}\subseteq [0,\infty]^X$ satisfies the conditions {\rm(R1)--(R3)}. Define a function $\delta: X\times 2^X\rightarrow [0, \infty]$ by \begin{equation}\delta(x,A)=\sup\{\phi(x)\mid \phi\in\mathcal{S}, \forall a\in A, \phi(a)=0\}.\end{equation}  Then $(X,\delta)$ is an approach space with   $\mathcal{S}$ being its set of regular functions. \end{prop}

Contractions between approach spaces can   be characterized in terms of regular functions.
\begin{prop}\label{contraction by regular frame} (\cite{RL97}) If $(X,\delta)$ and $(Y,\rho)$ are approach spaces and $f:X\rightarrow Y$ is a map, then $f$ is a contraction if and only if for each $\phi\in\mathcal{R}Y$, $\phi\circ f \in\mathcal{R}X$. \end{prop}

Since $\Omega(\mathbb{P})=([0,\infty],d_R)$, each regular function of an approach space $(X,\delta)$ is  a weight of the metric space $(X,\Omega(\delta))$. Given a metric space $(X,d)$,   the universal property of the map ${\rm id}:(X,d)\ra\Omega\circ\Gamma(X,d)$ entails that a map $\phi:X\ra[0,\infty]$ is a weight of $(X,d)$ if and only if it is a regular function of $\Gamma(X,d)$, as stated in the following conclusion.
\begin{prop}(\cite{RL97}, Proposition 3.1.9) \label{regular function in MAS} For a metric space $(X,d)$,  a function $\phi:X\rightarrow[0,\infty]$ is a weight of $(X,d)$  if and only if it is a regular function of the approach space $\Gamma(X,d)$. \end{prop}

\begin{defn}(\cite{BRC})   An approach prime of an approach space $(X,\delta)$ is a regular function $\phi$ subject to the following conditions: \begin{enumerate}[(1)] \item $\inf_{x\in X}\phi(x)=0$;  \item for all regular functions $\xi$ and $\psi$ of $(X,\delta)$, if $\min\{\xi, \psi\}\leq \phi$ then either $\xi\leq\phi$ or $\psi\leq\phi$. \end{enumerate} \end{defn}

For each element $x$ in an approach space $(X,\delta)$,  $\delta(-,\{x\})$ is an approach prime.
The following notion is  central  in this paper.

\begin{defn}(\cite{BRC}) An approach space $(X,\delta)$ is sober if for each approach prime  $\phi$ of $(X,\delta)$, there exists a unique $x\in X$ such that $\phi=\delta (-,\{x\})$.
\end{defn}

The approach space $\mathbb{P}$ is sober. This is proved in \cite{GVV}, Proposition 1.6. Another proof is contained in Proposition \ref{P is sober}.

 Write \[\omega: 2\rightarrow[0,\infty]\] for the map that sends $1$ in the quantale  $2$ to $0$ in $[0,\infty]$ and sends $0$ in   $2$ to $\infty$ in $[0,\infty]$.

 If $p:X\times X\rightarrow 2$ is an order on $X$, then the composite of \[\omega\circ p:X\times X\rightarrow 2\rightarrow[0,\infty]\] is a metric on $X$. Similarly, if   $c:X\times2^X\rightarrow 2$   satisfies  the conditions (C1)-(C4), then the composite of \[\omega\circ c:X\times2^X\rightarrow 2\rightarrow[0,\infty]\]  is an approach distance on $X$. These processes yield two full and faithful functors $\omega:{\sf Ord}\rightarrow{\sf Met}$ and   $\omega:{\sf Top}\rightarrow{\sf App}$. Both of them are denoted by the same symbol   since this will cause  no confusion. Approach spaces of the form $\omega(X)$  are said to be \emph{topological} \cite{RL97}.

 Write \[\iota:  [0,\infty]\rightarrow 2\] for the map that sends   $0$ in $[0,\infty]$ to $1$ in  the quantale $2$ and sends all $x$ in $(0,\infty]$ to $0$ in   $2$.

Given a metric $d:X\times X\rightarrow[0,\infty]$   on a set $X$,   the composite $\iota\circ d$ is an order on $X$, called the \emph{underlying order} of $d$. Given an approach distance $\delta:X\times2^X\rightarrow[0,\infty]$, the composite $\iota\circ \delta$ satisfies the conditions  (C1)-(C4), hence determines a topology on $X$, called the \emph{underlying topology} of $\delta$.
In this way, we obtain two (forgetful) functors: $\iota: {\sf Met}\rightarrow{\sf Ord}$ and $\iota: {\sf App}\rightarrow{\sf Top}$. It is easily seen that $\iota$ is right adjoint to $\omega$ (for both cases) and that the following diagrams are commutative: \[\bfig
\square<700,500>[{\sf Ord}` {\sf Top}` {\sf Met}  `{\sf App}; \Gamma` \omega ` \omega ` \Gamma]
\square(1500,0)/>`<-`<-`>/<700, 500>[{\sf Top} `{\sf Ord} `{\sf App}`{\sf Met} ;  \Omega `\iota`\iota`\Omega]
\efig\]

Both $\omega: (2, \wedge)\rightarrow([0,\infty]^{\rm op},+)$ and $\iota:([0,\infty]^{\rm op},+)\rightarrow (2, \wedge)$ are closed maps between quantales \cite{Ro90} (or, lax functors \cite{HST,Lawvere73} if quantales are treated as monoidal closed categories). So, both   $\omega:{\sf Ord}\rightarrow{\sf Met}$ and $\iota: {\sf Met}\rightarrow{\sf Ord}$ are examples of the change-of-base functors in enriched category theory \cite{Lawvere73}.
 The following conclusion shows  that the notion of sober approach spaces extends that of sober topological spaces.

\begin{prop}(\cite{BRC}) \label{top sober} A topological space $X$ is sober if and only if $\omega(X)$ is a sober approach space.   The underlying topology of a sober approach space is   sober.   \end{prop}

\section{Sobrification of approach spaces}

The sobrification of an approach space $(X,\delta)$ is constructed in \cite{BRC} as the spectrum of the approach frame of regular functions of $(X,\delta)$. In this section, we present a description of this construction without resort to the notion of approach frames. This description will be useful in subsequent sections.

For an approach space  $(X,\delta)$, let
\[\widehat{X}=\{\phi\in \mathcal{R}X\mid \phi~\rm{is~an ~approach~prime}\}.\]
For each $\xi\in \mathcal{R}X$,   define a map  $\widehat{\xi}: \widehat{X}\rightarrow [0,\infty]$ by
\[  \widehat{\xi}(\phi)=\sup_{x\in X}d_L(\phi(x),\xi(x))=\inf\{\alpha\in[0,\infty]\mid \xi\leq \phi+\alpha\}.\]

\begin{lem}\label{widehat propoties} Let $(X,\delta)$ be an approach space.
\begin{enumerate}[(1)]
\item For all $\xi\in\mathcal{R}X$ and $a\in X$, $\widehat{\xi}(\delta(-,\{a\}))=\xi(a)$.
\item For all $\xi,\psi \in\mathcal{R}X$,   $\xi\leq \psi\Leftrightarrow \widehat{\xi}\leq\widehat{\psi}$.
\item For all $\xi  \in\mathcal{R}X$ and $\phi\in\widehat{X}$, $\widehat{\xi}(\phi)=0 \Leftrightarrow \xi\leq \phi$.
  \item For every subset $\{\xi_i\}_{i\in I}$ of $\mathcal{R}X$, $\widehat{\sup \xi_i}=\sup \widehat{\xi_i}$.
  \item For all $\xi, \psi\in \mathcal{R}X$, $\widehat{\min\{\xi, \psi\}}=\min\{\widehat{\xi},\widehat{\psi}\}$.
  \item For all $\xi\in\mathcal{R}X$ and $\alpha\in [0,\infty]$, $\widehat{\xi+\alpha}=\widehat{\xi}+\alpha$ and $\widehat{\xi\ominus \alpha}= \widehat{\xi}\ominus \alpha $.
\end{enumerate}\end{lem}
\begin{proof} We check (1)  and (5) for example.

(1) On one hand,  by definition of $\widehat{\xi}$, \[\widehat{\xi}(\delta(-,\{a\}))=\sup_{x\in X}d_L(\delta(x,\{a\}),\xi(x))\geq d_L(\delta(a,\{a\}),\xi(a))=\xi(a).\] On the other hand,  since $\xi:(X,\delta)\rightarrow\mathbb{P}$ is a contraction, it follows that $\xi(x)-\delta(x,\{a\})\leq \xi(a)$  for all $x\in X$, hence
$$\widehat{\xi}(\delta(-,\{a\}))=\sup_{x\in X}d_L(\delta(x,\{a\}),\xi(x))\leq \xi(a).$$

(5) That $\widehat{\min\{\xi, \psi\}}\leq\min\{\widehat{\xi},\widehat{\psi}\}$ is obvious. It remains to check that $\min\{\widehat{\xi}(\phi),\widehat{\psi}(\phi)\}\leq\widehat{\min\{\xi, \psi\}}(\phi)$ for all $\phi\in \widehat{X}$.
By definition, \[\widehat{\min\{\xi, \psi\}}(\phi)=\inf\{\alpha\mid \min\{\xi, \psi\}\leq \phi+\alpha\}.\]  Since $\phi$ is an approach prime,  for each $\alpha\in[0,\infty]$, if $\min\{\xi,\psi\}\leq \phi+\alpha$, then either $\xi\ominus\alpha\leq \phi $ or $\psi\ominus\alpha\leq \phi$, it follows that either $\xi\leq \phi+\alpha$ or $\psi\leq \phi+\alpha$, so, $\min\{\widehat{\xi}(\phi),\widehat{\psi}(\phi)\}\leq \alpha$, hence $\min\{\widehat{\xi}(\phi),\widehat{\psi}(\phi)\} \leq\widehat{\min\{\xi, \psi\}}(\phi).$
\end{proof}

Given an approach space $(X,\delta)$,  the set $\{\widehat{\xi}\mid \xi\in\mathcal{R}X\}$ satisfies the conditions in Proposition \ref{regular functions}, hence it determines an approach distance $\widehat{\delta}$  on $\widehat{X}$ via \begin{equation}\label{widehat distance*}\widehat{\delta}(\phi,A)= \sup\{\widehat{\psi}(\phi)\mid \psi\in\mathcal{R}X,\forall \xi\in A,\widehat{\psi}(\xi)=0\} \end{equation} for all $\phi\in\widehat{X}$ and $A\subseteq \widehat{X}$. In particular,  for all $\phi,\xi\in \widehat{X}$, \begin{equation}\label{widehat distance}\widehat{\delta}(\phi,\{\xi\})= \sup\{\widehat{\psi}(\phi)\mid \psi\in\mathcal{R}X, \widehat{\psi}(\xi)=0\} =\widehat{\xi}(\phi). \end{equation}

Define a map
\[ \eta_X:(X,\delta)\rightarrow (\widehat{X},\widehat{\delta}) \] by $\eta_X(x)=\delta(-,\{x\})$. Then for all $x\in X$ and $A\subseteq X$, \begin{align*}\widehat{\delta}(\eta_X(x), \eta_X(A))&= \sup\{\widehat{\psi}(\delta(-,\{x\}))\mid \psi\in \mathcal{R}X, \forall a\in A, \widehat{\psi}(\delta(-,\{a\}))=0\} \\ &= \sup\{\psi(x)\mid \psi\in \mathcal{R}X, \forall a\in A, \psi(a)=0\}\\ &=\delta(x,A).\end{align*} This shows that $\eta_X:(X,\delta)\rightarrow (\widehat{X},\widehat{\delta})$ is an  isometric map.

It is clear that $(X,\delta)$ is sober if and only if $\eta_X$ is bijective, hence an isomorphism in  ${\sf App}$.

\begin{thm}\label{sobrification}Let $(X,\delta)$ be  an approach space. \begin{enumerate}[(1)]  \item $(\widehat{X},\widehat{\delta})$ is a sober approach space.
\item For each contraction $f$ from $(X,\delta)$ to a sober approach space $(Y,\rho)$, there is a unique contraction $\overline{f}:(\widehat{X},\widehat{\delta})\rightarrow (Y,\rho)$ such that $f=\overline{f}\circ\eta_X$. \end{enumerate} \end{thm}
\begin{proof} (1) We must show that each approach prime  of $(\widehat{X},\widehat{\delta})$ is of the form $\widehat{\delta}(-,\{\phi\})$ for a unique approach prime $\phi$ of $(X,\delta)$. Uniqueness of $\phi$ is clear since $$\widehat{\delta}(\delta(-,x),\{\phi\})=\phi(x)$$  for all $x\in X$. It remains to check the existence.

By definition, each approach prime (indeed, each regular function) on $(\widehat{X},\widehat{\delta})$ is of the form $\widehat{\xi}$ for some $\xi\in\mathcal{R}X$.
Given an approach prime $\widehat{\xi}$ of $(\widehat{X},\widehat{\delta})$, if we could show that   $\xi$ is an approach prime of $(X,\delta)$, then we would obtain $\widehat{\xi}=\widehat{\delta}(-,\{\xi\})$  by virtue of Equation (\ref{widehat distance}), proving the existence.  So, it suffices to show that if $\widehat{\xi}$ is an approach prime  of $(\widehat{X},\widehat{\delta})$, then $\xi$ is an approach prime   of $(X,\delta)$.

(a) $\inf_{x\in X}\xi(x)=0$. Given $\varepsilon>0$, since $\widehat{\xi}$ is an approach prime, there is   $\phi\in\widehat{X}$ such that \[ \widehat{\xi}(\phi)=\sup_{x\in X}d_L(\phi(x),\xi(x))<\varepsilon.\] Since $\phi$ is an approach prime of $(X,\delta)$, there exists some $x_0$ such that $\phi(x_0)<\varepsilon$. Thus, $\xi(x_0)<2\varepsilon$, so, $\inf_{x\in X}\xi(x)=0$ by arbitrariness of $\varepsilon$.

(b) Suppose that $\phi, \psi\in\mathcal{R}X$  and that $\min\{\phi, \psi\}\leq \xi$. Then  $\min\{\widehat{\phi},\widehat{\psi}\}=\widehat{\min\{\phi, \psi\}}\leq \widehat{\xi}$. Since $\widehat{\xi}$ is an approach prime,   either $\widehat{\phi}\leq \widehat{\xi}$ or $\widehat{\psi}\leq \widehat{\xi}$, it follows that either $\phi\leq \xi$ or $\psi\leq \xi$ by Lemma \ref{widehat propoties}(2).

Therefore, $\xi\in\widehat{X}$, as desired.

(2) Suppose $(Y,\rho)$ is a sober approach space, $f:(X,\delta)\rightarrow(Y,\rho)$ is a contraction.  We show that there is a unique contraction $\overline{f}:(\widehat{X},\widehat{\delta})\rightarrow (Y,\rho)$ such that $f=\overline{f}\circ\eta_X$.

\textbf{Existence}. For each $\phi\in \widehat{X}$, let \[ f^\dag(\phi) =\sup\{\psi\in\mathcal{R}Y\mid\psi\circ f\leq \phi\}.\] That is, $\psi\leq  f^\dag(\phi)\iff \psi\circ f\leq \phi$ for all $\psi\in\mathcal{R}Y$. We claim that $ f^\dag(\phi)$ is an approach prime of $(Y,\rho)$.

Given $\varepsilon>0$, there exists some $x_\varepsilon\in X$ such that $\phi(x_\varepsilon)<\varepsilon$. Let $y_\varepsilon=f(x_\varepsilon)$. Then $\psi(y_\varepsilon)< \varepsilon$ whenever $\psi\circ f\leq \phi$, it follows that $ f^\dag(\phi)(y_\varepsilon)\leq\varepsilon$. Therefore, $\inf_{y\in Y} f^\dag(\phi)(y)=0$.

Suppose that $\psi,\xi\in\mathcal{R}Y$ and $\min\{\psi,\xi\}\leq f^\dag(\phi)$. Since \[ \min\{\psi\circ f,\xi\circ f\}=\min\{\psi,\xi\}\circ f \leq\phi,\] we obtain that either $\psi\circ f\leq \phi$ or $\xi\circ f\leq \phi$, hence either $\psi\leq f^\dag(\phi)$ or $\xi\leq f^\dag(\phi)$.

Therefore, $ f^\dag(\phi)$ is an approach prime of $(Y,\rho)$. Since $(Y,\rho)$ is sober, there is a unique $y\in Y$ such that $ f^\dag(\phi)=\rho(-,\{y\})$. Define $\overline{f}(\phi)$ to be this $y$. We claim that $\overline{f}:\widehat{X}\rightarrow Y$ satisfies the requirement.

(a) $\overline{f}:(\widehat{X},\widehat{\delta})\rightarrow (Y,\rho)$ is a contraction.   By Proposition \ref{contraction by regular frame}, it is sufficient to show that for each $\psi\in\mathcal{R}Y$, $\psi\circ \overline{f}$ is a regular function of $(\widehat{X},\widehat{\delta})$. Since $f:(X,\delta)\rightarrow(Y,\rho)$ is a contraction, $\psi\circ f$ is a regular function of $(X,\delta)$. If we could show that $\psi\circ \overline{f}= \widehat{\psi\circ f}$, then $\psi\circ \overline{f}$ would be a regular function of $(\widehat{X},\widehat{\delta})$, as desired.

For each $\phi\in\widehat{X}$,
 \begin{align*}
 \psi\circ \overline{f}(\phi) &=\widehat{\psi}(\rho(-,\{\overline{f}(\phi)\})) &({\rm Lemma}~\ref{widehat propoties}(1))\\
 & = \widehat{\psi}( f^\dag(\phi)) &({\rm definition~of}~\overline{f})\\
 &= \inf\{\alpha\mid \psi\leq  f^\dag(\phi) +\alpha\} \\
 &= \inf\{\alpha\mid \psi\ominus \alpha\leq f^\dag(\phi) \} \\
 & = \inf\{\alpha\mid \psi\circ f\ominus \alpha\leq \phi  \} \\
 & = \inf\{\alpha\mid \psi\circ f\leq \phi+\alpha\} \\
 & = \widehat{\psi\circ f}(\phi). \end{align*}

(b)  $f=\overline{f}\circ\eta_X$. For each $x\in X$, $\overline{f}\circ\eta_X(x)$ is the unique element in $Y$ such that $\rho(-,\{\overline{f}\circ\eta_X(x)\}) =f^\dag(\delta(-,\{x\}))$, so, it suffices to check that $\rho(-,\{f(x)\}) =f^\dag(\delta(-,\{x\}))$. On one hand, since for any $x'\in X$, \[ \rho(-,\{f(x)\})\circ f(x')= \rho(f(x'),\{f(x)\}) \leq \delta(x',\{x\}),\] it follows that $\rho(-,\{f(x)\})\circ f\leq  \delta(-,\{x\})$, hence $\rho(-,\{f(x)\})\leq f^\dag(\delta(-,\{x\}))$. On the other hand, suppose that $\psi\in\mathcal{R}Y$ and $\psi\circ f\leq \delta(-,\{x\})$. Then for any $y\in Y$, \begin{align*}\psi(y)\leq\psi\circ f(x) +\rho(y,\{f(x)\}) \leq \delta(x,\{x\})+\rho(y,\{f(x)\})=\rho(y,\{f(x)\}), \end{align*} showing that $f^\dag(\delta(-,\{x\}))\leq \rho(-,\{f(x)\})$.

\textbf{Uniqueness}. Suppose $g: (\widehat{X},\widehat{\delta})\rightarrow(Y,\rho)$ is a contraction with $g\circ\eta_X=f$. We show that  for each $\phi\in \widehat{X}$, $g(\phi) =\overline{f}(\phi)$, i.e., $ f^\dag(\phi)=\rho(-,\{g(\phi)\})$.

On one hand, for each $x\in X$, since $g$ is a contraction, one has $$\rho(-,\{g(\phi)\})\circ f(x)=\rho(g\circ\eta_X(x),\{g(\phi)\})\leq \widehat{\delta}(\eta_X(x),\{\phi\})= \widehat{\phi}(\eta_X(x))=\phi(x),$$
so,  $\rho(-,\{g(\phi)\})\circ f\leq \phi$, hence $\rho(-,\{g(\phi)\})\leq  f^\dag(\phi)$.

On the other hand, for every $\psi\in\mathcal{R}Y$ with $\psi\circ f\leq \phi$, since $g$ is a contraction,  $\psi\circ g$ is a regular function of $(\widehat{X},\widehat{\delta})$, hence there exists some $\xi\in \mathcal{R}X$ such that $\psi\circ g=\widehat{\xi}$. Then  \begin{align*}\psi\circ f\leq \phi &~\Rightarrow~ \psi\circ g\circ\eta_X\leq\phi \\ &~\Rightarrow~  \widehat{\xi}\circ\eta_X\leq\phi \\ &~\Rightarrow~ \forall x\in X, \xi(x)\leq\phi(x) \\ &~\Rightarrow~ \widehat{\xi}(\phi)=0. \end{align*}
Since $\psi:(Y,\rho)\rightarrow\mathbb{P}$ is a contraction, it follows that for each $y\in Y$,
$$\psi(y)\leq \psi\circ g(\phi)+\rho(y,\{g(\phi)\})= \widehat{\xi}(\phi)+\rho(y,\{g(\phi)\}) =\rho(y,\{g(\phi)\}).$$ This proves the inequality $ f^\dag(\phi) \leq\rho(-,\{g(\phi)\})$.   \end{proof}

Let ${\sf SobApp}$ denote the full subcategory of ${\sf App}$ consisting of sober approach spaces. The  universal property of $(\widehat{X},\widehat{\delta})$ gives rise to a functor
$$s:{\sf App}\rightarrow{\sf SobApp},\quad s(X,\delta)= (\widehat{X},\widehat{\delta})$$ that is left adjoint to the inclusion functor ${\sf SobApp}\rightarrow{\sf App}$. The sober approach space $s(X,\delta)$ is called the \emph{sobrification} of $(X,\delta)$.

\section{Yoneda completion of metric spaces}
A net $\{x_\lam\}$ in a metric space $(X,d)$ is forward Cauchy \cite{BvBR1998,Wagner97} if \[\inf_\lam\sup_{\nu\geq\mu\geq\lam}d(x_\mu,x_\nu)=0.\] %That is, $\{x_\lam\}$ is forward Cauchy \cite{BvBR1998,Wagner97} if, for any $\varepsilon>0$, there exists $\lambda_0$ such that $d(x_{\mu},x_{\nu}) <\varepsilon$ whenever $\nu\geq\mu\geq \lambda_0$.

\begin{defn}(\cite{BvBR1998,Wagner97})
Let $\{x_\lam\}$  be a net in a metric space $(X,d)$. An element   $x\in X$  is a Yoneda limit (a.k.a. liminf) of $\{x_\lam\}$ if  for all $y\in X$, \[d(x,y)= \inf_\lambda\sup_{\sigma\geq\lambda}d(x_{\sigma},y).\] \end{defn}

Yoneda limits are not necessarily unique. However, if both $x$ and $y$ are   Yoneda limit  of a  net $\{x_\lambda\}$, then $d(x,y)=d(y,x)=0$. So, Yoneda limits in separated metric spaces are unique.

\begin{defn} (\cite{BvBR1998,Wagner97}) A metric space  is Yoneda complete if each forward Cauchy net in it has a Yoneda limit. \end{defn}
A non-expansive map $f:(X,d)\rightarrow(Y,p)$ is \emph{Yoneda continuous} if it preserves Yoneda limits in the sense that if $a$ is a Yoneda limit of a forward Cauchy net $\{x_\lam\}$ then $f(a)$ is a Yoneda limit of $\{f(x_\lam)\}$.

\begin{exmp}\label{d_L} Consider the metric space $([0,\infty],d_L)$.  If $\{x_\lambda\}$ is a forward Cauchy net in $([0,\infty],d_L)$, then  $\{x_\lambda\}$ is either an eventually constant net with value $\infty$ or eventually a Cauchy net of real numbers in the usual sense.  In the first case, $\infty$ is a Yoneda limit of $\{x_\lambda\}$; in the second case, the limit of the Cauchy net $\{x_\lambda\}$ is a Yoneda limit of $\{x_\lambda\}$. Thus, $([0,\infty],d_L)$ is Yoneda complete.

It is easily seen that for each forward Cauchy net $\{x_\lambda\}$  in a metric space $(X,d)$, $\{d(x,x_\lam)\}$ is a forward Cauchy net  in $([0,\infty],d_L)$ for all $x\in X$. In particular, \begin{equation}\label{forward cauchy nets converge} \inf_\lam\sup_{\si\geq\lam}d(x,x_\si)= \sup_\lam\inf_{\si\geq\lam}d(x,x_\si).\end{equation} \end{exmp}

\begin{exmp}\label{d_R} Consider the metric space $([0,\infty],d_R)$.  A net $\{x_\lambda\}$ in $[0,\infty]$ is almost increasing if for each $\varepsilon>0$, there is some $\lam$ such that $x_\mu-x_\nu\leq\varepsilon$ whenever $\nu\geq\mu\geq \lam$. It is clear that every almost increasing net is forward Cauchy in $([0,\infty],d_R)$. Furthermore, if a net $\{x_\lambda\}$ is   forward Cauchy  in $([0,\infty],d_R)$, then $\{x_\lambda\}$ is either an almost increasing net that tends to infinity or a Cauchy net in the usual sense.  In the first case, $\infty$ is a Yoneda limit of $\{x_\lambda\}$ in $([0,\infty],d_R)$; in the second case, the limit of the Cauchy net $\{x_\lambda\}$ is a Yoneda limit of $\{x_\lambda\}$ in $([0,\infty],d_R)$. Thus, $([0,\infty],d_R)$ is Yoneda complete.

 The metric space $([0,\infty),d_R)$ is not  Yoneda complete, but $([0,\infty),d_L)$ is.
\end{exmp}

\begin{prop}The underlying order of a Yoneda complete metric space is directed complete. \end{prop}
\begin{proof} Let $(X,d)$ be a Yoneda complete metric space, $\leq_d$  be the underlying order of $(X, d)$, and $D$ be  a directed subset in $(X,\leq_d)$. Regard $D$ as a net $\{x\}_{x\in D}$ in $(X,d)$  in the obvious way. By definition of $\leq_d$ we have $d(x,y)=0$  whenever $x\leq_d y$. Thus, $\{x\}_{x\in D}$ is a forward Cauchy net in $(X,d)$. Let $a$ be a Yoneda limit of $\{x\}_{x\in D}$. We show that $a$ is a join of $D$ in $(X,\leq_d)$.

Since $a$ is a Yoneda limit of $\{x\}_{x\in D}$, it holds that $d(a,y)=\inf_{x\in D}\sup_{z\geq_d x}d(z,y)$ for each $y\in X$, in particular, $\inf_{x\in D}\sup_{z\geq_d x}d(z,a)=d(a,a)=0$.
Thus, for each $\varepsilon>0$ there exists $x_\varepsilon\in D$ such that $d(x,a)<\varepsilon$ for all $x\in D$ with $x\geq_d x_\varepsilon$.

For a fixed $x\in D$,  let $y$ be an upper bound of $x$ and $x_\varepsilon$ in $D$. Then $$d(x,a)\leq d(x,y)+d(y,a)\leq \varepsilon.$$ Therefore, $d(x,a)=0$ by arbitrariness of $\varepsilon$, showing that $a$ is an upper bound of $D$.

Let $z$ be another upper bound of $D$. Then for all $y\in D$ we have $d(y,z)=0$. So, $$d(a,z)=\inf_{x\in D}\sup_{y\geq_d x}d(y, z)=0,$$ showing that $a\leq_d z$. This proves that $a$ is a join of $D$ in $(X,\leq_d)$.
\end{proof}

For each weight $\phi$ and each coweight $\psi$ of a metric space $(X,d)$, the tensor product of $\phi$ and $\psi$ \cite{SV2005} (a special case of composition of bimodules in \cite{Lawvere73}) is an element in   $[0,\infty]$, given by \[\phi\otimes \psi=\inf_{x\in X}(\phi(x)+\psi(x)).\]

Let $\phi$ and $\psi$ be a weight and a coweight of a metric space $(X,d)$, respectively. We say that   $\phi$ is  a  right adjoint of $\psi$ (or,  $\psi$ is a  left adjoint of $\phi$) if $\phi\otimes \psi=0$ and $\phi(x)+\psi(y)\geq d(x,y)$ for all $x,y\in X$.   This notion is a special case of adjoint bimodules in enriched category theory \cite{Lawvere73,St05}. So, the left adjoint of a weight, if exists, is unique.

\begin{defn}Let $(X,d)$ be a metric space, $\phi$ a weight of $(X,d)$. \begin{enumerate}[(1)] \item (\cite{Lawvere73}) $\phi$ is a Cauchy weight if it has a left adjoint. \item (\cite{SV2005}) $\phi$ is a flat weight if $\inf_{x\in X}\phi(x)=0$ and
$\phi\otimes\max\{\psi_1,\psi_2\}=\max\{\phi\otimes \psi_1,\phi\otimes \psi_2\}$ for all coweights $\psi_1,\psi_2$ of $(X,d)$. \end{enumerate} \end{defn}

Each representable weight $d(-,x)$ is Cauchy, since it is right adjoint to the coweight $d(x,-)$. Following Lawvere \cite{Lawvere73}, we say that a metric space is Cauchy complete if it is separated and all of its Cauchy weights are representable. In the realm of separated and symmetric metric spaces, this notion of Cauchy completeness agrees with the traditional one, namely, every Cauchy sequence converges.

If   $\phi$ is a Cauchy weight of $(X,d)$, it is easy to check that its left adjoint is  given by \begin{equation}\label{left adjoint}\phi^\vdash(x)=\overline{d}(\phi,d(-,x)). \end{equation}

\begin{lem} \label{Cauchy weight} Let $\phi$ be a Cauchy weight of a metric space $(X,d)$ and $\phi^\vdash$ be its left adjoint.
\begin{enumerate}[(1)]
\item For each coweight $\psi$ of $(X,d)$, $\phi\otimes \psi=\sup_{y\in X}d_L(\phi^\vdash(y),\psi(y)).$
\item For each  weight $\xi$ of $(X,d)$, $\xi\otimes \phi^\vdash =\overline{d}(\phi,\xi)$.
\item For each non-empty set $\{\psi_i\}$ of coweights  of $(X,d)$, $\phi\otimes\sup_i\psi_i=\sup_i(\phi\otimes \psi_i)$. In particular, $\phi$ is  flat.
\item For each non-empty set $\{\xi_i\}$ of weights  of $(X,d)$,  $\inf_i \overline{d}(\phi, \xi_i) =\overline{d}(\phi,\inf_i\xi_i).$ \end{enumerate} \end{lem}

\begin{proof} The formulas in (1) and (2) are a special case of   2(d) and 2(e) in Stubbe \cite{St05},  Lemma 2.2 that hold  for all quataloids. We include here a direct verification   for convenience of the reader.

 (1) For each $y\in Y$,  \begin{align*} \phi\otimes \psi+ \phi^\vdash(y) &= \inf_{x\in X}(\phi(x)+\psi(x)) + \phi^\vdash(y) \\ & =\inf_{x\in X}(\phi(x)+\psi(x)  + \phi^\vdash(y)) \\ &\geq  \inf_{x\in X}(d(x,y)+\psi(x)) \\ &\geq \psi(y), \end{align*} it follows that $\phi\otimes \psi\geq \sup_{y\in X}d_L(\phi^\vdash(y),\psi(y))$. To see the converse inequality, take   $a\in[0,\infty]$ with $a\geq \sup_{y\in X}d_L(\phi^\vdash(y),\psi(y))$. Then we have  \begin{align*}&\forall y,  d_L(\phi^\vdash(y),\psi(y)) \leq a \\ \Rightarrow ~&  \forall y, \psi(y) \leq \phi^\vdash(y)+a \\ \Rightarrow ~&  \forall y, \phi(y)+\psi(y)\leq \phi(y)+\phi^\vdash(y)+a\\ \Rightarrow ~&  \phi\otimes \psi\leq a.\end{align*} This proves that $\phi\otimes \psi\leq \sup_{y\in X}d_L(\phi^\vdash(y),\psi(y))$.

(2) Similar to (1).

(3) An immediate consequence of (1).

(4) An immediate consequence of (2).\end{proof}

Let $f:(X,d)\rightarrow(Y,p)$ be a non-expansive map between metric spaces. If $\phi$ is a weight of $(X,d)$ then $f(\phi):Y\rightarrow[0,\infty]$, given by \[f(\phi)(y) =\inf_{x\in X}(\phi(x)+p(y,f(x))),\] is a weight of $(Y,p)$. If $\psi$ is a weight (coweight, resp.) of $(Y,p)$  then $\psi\circ f$ is a weight (coweight, resp.) of $(X,d)$.

\begin{prop}\label{image of Cauchy weight} Let $f:(X,d)\rightarrow(Y,p)$ be a non-expansive map between metric spaces, $\phi$ a weight of $(X,d)$.   \begin{enumerate}[(1)]
\item If $\phi$ is flat then so is $f(\phi)$.
\item If $\phi$ is Cauchy then so is $f(\phi)$.
\item If $f(\phi)$ is Cauchy and $f$ is an isometric  map  then $\phi$ is Cauchy. \end{enumerate} \end{prop}

\begin{proof}(1) First, $\inf_{y\in Y}f(\phi)(y) \leq \inf_{y=f(x)}f(\phi)(y)=\inf_{x\in X}\phi(x)=0$. Second, it is easy to check that for each coweight $\psi$ of $(Y,p)$ it holds that $$f(\phi)\otimes\psi=\phi\otimes(\psi\circ f).$$
Therefore, for all coweights $\psi_1, \psi_2$ of $(Y,p)$, we have
\begin{align*}
 f(\phi)\otimes\max\{\psi_1, \psi_2\}&=\phi\otimes(\max\{\psi_1, \psi_2\}\circ f)\\
 &= \phi\otimes(\max\{\psi_1\circ f, \psi_2\circ f\}) \\
 &=\max\{\phi\otimes(\psi_1\circ f),\phi\otimes(\psi_2\circ f)\} &(\phi~{\rm is~flat})\\
 &=\max\{f(\phi)\otimes\psi_1,f(\phi)\otimes\psi_2\},
\end{align*} showing that $f(\phi)$ is flat.

(2) If $\phi^\vdash$ is a left adjoint of $\phi$, then the coweight $\psi:Y\rightarrow[0,\infty]$, given by $\psi(y) =\inf_{x\in X}(\phi^\vdash(x)+p(f(x),y))$, is a left adjoint of $f(\phi)$.

(3) We leave it to the reader to check that if $\psi$ is  a left adjoint of $f(\phi)$, then $\psi\circ f$ is   a left adjoint of $\phi$. \end{proof}

The following proposition is contained in Vickers \cite{SV2005},  Proposition 7.9 and Theorem 7.15. An extension to generalized partial metric spaces can be found in \cite{LLZ16}, Proposition 7.4.

\begin{prop} (\cite{SV2005}) \label{flat weight} Let $(X,d)$ be a metric space.  Then for each function $\phi:X\rightarrow[0,\infty]$,  the following   are equivalent:
\begin{enumerate}[\rm(1)] \item $\phi$ is a flat weight of $(X,d)$.
\item  $\phi$ is a weight of $(X,d)$   satisfying the following conditions: \begin{enumerate}
\item[\rm(a)]$\inf_{x\in X}\phi(x)=0$;
\item[\rm(b)]if $\phi(x_i)<\varepsilon_i\ (i=1,2)$, then there is some $y\in X$ and $\varepsilon>0$ such that $f(y)<\varepsilon$ and that $d(x_i,y)+\varepsilon<\varepsilon_i\ (i=1,2)$.
\end{enumerate}
\item There is a forward Cauchy net $\{x_\lam\}$ in $(X,d)$ such that $\phi=\inf_\lam\sup_{\sigma\geq\lam} d(-,x_\sigma)$.\end{enumerate}
\end{prop}
For a metric space $(X,d)$, let
\[(\CF X,\overline{d})\] be the subspace of $(\CP X,\overline{d})$ consisting of flat weights. Define
\[\y_X:(X,d)\rightarrow (\CF X,\overline{d})\] by $\y_X(x)= d(-,x)$. Then
$\y_X$ is an isometric map.

\begin{prop}\label{supremum}   A metric space $(X,d)$  is Yoneda complete if  and only if  for each flat weight $\phi$ of $(X,d)$, there is some $a\in X$ such that  for all $y\in X$, \begin{equation}\label{colim} \overline{d}(\phi,\y_X(y))=d(a,y).\end{equation} \end{prop}
\begin{proof} By Lemma 46 in \cite{FSW}, for each forward Cauchy net $\{x_\lam\}$    in   $(X,d)$, an element $a$ in $X$ is a Yoneda limit of  $\{x_\lam\}$ if and only if $\overline{d}(\phi,\y_X(y))=d(a,y)$ for all $y\in X$, where $\phi$ is the weight of $(X,d)$ given by $\phi=\inf_\lam\sup_{\sigma\geq\lam} d(-,x_\sigma)$. The  conclusion follows immediately from a combination of this fact and Proposition \ref{flat weight}. \end{proof}

An element $a$ satisfying Equation (\ref{colim}) is called, in enriched category theory, a colimit of the identity $(X,d)\rightarrow(X,d)$ weighted by $\phi$ \cite{Kelly,KS05,Ru}. In this paper, we simply say that $a$ is a \emph{colimit} of $\phi$ and write $a=\colim\phi$. The above proposition says that a metric space $(X,d)$  is Yoneda complete if  and only if every flat weight of $(X,d)$  has a colimit.

The following conclusion is contained in Vickers \cite{SV2005}, Proposition 7.14 and Theorem 7.15. It implies that for each metric space $(X,d)$, the metric space $(\CF X,\overline{d})$ is Yoneda complete.

\begin{thm}(\cite{SV2005})\label{Yoneda complete} Let  $(X,d)$ be a metric space. Every forward Cauchy net $\{\phi_\lam\}$ in the metric space $(\CP X,\overline{d})$ has a Yoneda limit given by $\inf_\lam\sup_{\lam\leq\mu}\phi_\mu$; the subspace $(\CF X,\overline{d})$ is closed in $(\CP X,\overline{d})$ with respect to  Yoneda limits of forward Cauchy nets. \end{thm}

From the point of view of category theory, a combination of Proposition \ref{image of Cauchy weight}(1), Proposition \ref{flat weight}, Proposition \ref{supremum} and Theorem \ref{Yoneda complete} says that flat weights form a saturated class of weights \cite{KS05,LZ07} on metric spaces. As pointed out to us by the referee,  the saturatedness of the class of flat weights is a special case of a general result in enriched category theory, namely, Proposition 5.4 in Kelly and Schmidt \cite{KS05}. The space $(\CF X,\overline{d})$ has the following universal property: for each non-expansive map $f$ from $(X,d)$ to a separated and Yoneda complete metric space $(Y,p)$, there exists a unique Yoneda continuous map $f^*:(\CF X,\overline{d}) \rightarrow(Y,p)$ such that $f=f^*\circ\y_X$. This universal property is also a special case of a result  in  \cite{Kelly,KS05} about cocompletion with respect to saturated classes of weights. Because of this universal property, $(\CF X,\overline{d})$ is called the \emph{Yoneda completion} of $(X,d)$. The subspace of $(\CF X,\overline{d})$  consisting of Cauchy weights is the \emph{Cauchy completion} of $(X,d)$ \cite{Lawvere73}. %The space $(\CF X,\overline{d})$ is isomorphic to the Yoneda completion described in \cite{KS2002} in terms of forward Cauchy nets.

\begin{exmp}The Yoneda completion of $([0,\infty),d_R)$ is $([0,\infty],d_R)$. \end{exmp}

\section{Sobrification of  metric approach spaces}
In this section, we show that the specialization metric space of the sobrification of a metric approach space $\Gamma(X,d)$ coincides with the Yoneda completion of $(X,d)$.

\begin{lem} \label{approach prime} Let $(X,\delta)$ be an approach space. If $\{x_\lam\}$ is a forward Cauchy net  in $(X,\Omega(\delta))$, then the function \[\phi: X\rightarrow[0,\infty], \quad \phi(x)= \sup_\lam\delta(x,A_\lam),\]  is an approach prime of $(X,\delta)$, where  $A_\lam=\{x_\sigma\mid \sigma\geq\lam\}$.
\end{lem}

\begin{proof}For simplicity, we write $d$ for the metric $\Omega(\delta)$.
We prove the conclusion in three steps.

\textbf{Step 1}.  $\phi$ is a regular function of $(X,\delta)$. This follows from Proposition \ref{regular functions}(R1) and the fact that   $\delta(-,A_\lam)$ is a regular function for each $\lam$.

\textbf{Step 2}.  $\inf_{x\in X}\phi(x)=0$. For any $\varepsilon>0$, there exists $\lam_0$ such that $d(x_\mu,x_\nu)<\varepsilon$ whenever $\nu\geq\mu\geq\lam_0$. Then for all $\lam$,  $\delta(x_{\lam_0},A_\lam)\leq \inf_{\sigma\geq\lam_0,\lam}d(x_{\lam_0},x_\sigma) <\varepsilon,$ hence $\phi(x_{\lam_0})=\sup_\lam \delta(x_{\lam_0},A_\lam)\leq\varepsilon$. This shows that $\inf_{x\in X}\phi(x)=0$.

\textbf{Step 3}.  For any regular functions $\psi$ and $\xi$ of $(X,\delta)$, if $\min\{\psi,\xi\}\leq \phi$   then either $\psi\leq \phi$ or $\xi\leq \phi$. If not, there exist $x_1$ and $x_2$ such that $\psi(x_1)>\phi(x_1)$ and $\xi(x_2)>\phi(x_2)$. Take   $\varepsilon>0$ with $\psi(x_1)-\phi(x_1)>\varepsilon$ and $\xi(x_2)-\phi(x_2)>\varepsilon$, i.e.,
$$\psi(x_1)-\sup_\lam\delta(x_1,A_\lam)>\varepsilon, \quad  \xi(x_2)-\sup_\lam\delta(x_2,A_\lam)>\varepsilon.$$ Since $\psi,\xi:(X,\delta)\rightarrow\mathbb{P}$ are contractions,   for every $\lam$, it holds that \[\delta(x_1,A_{\lam})\geq \psi(x_1)-\sup \psi(A_{\lam}),\quad \delta(x_2,A_{\lam})\geq \xi(x_2)-\sup \xi(A_{\lam}),\] hence \[\sup \psi(A_{\lam})\geq\psi(x_1)- \delta(x_1,A_{\lam})>\varepsilon,\quad \sup \xi(A_{\lam})\geq\xi(x_2)- \delta(x_2,A_{\lam})>\varepsilon.\]
By arbitrariness of $\lam$ and the forward Cauchyness of  $\{x_\lambda\}$, there exists some $\sigma$ such that $\psi(x_\sigma)>\varepsilon$, $\xi(x_\sigma)>\varepsilon$,   and that $d(x_\sigma,x_\nu)<\varepsilon$ whenever $\sigma\leq\nu$. Then \[\phi(x_\sigma)=\sup_\lam \delta(x_\sigma,A_\lam)\leq \sup_\lam\inf_{\nu\geq\lam,\sigma}d(x_\sigma, x_\nu) \leq \varepsilon,\]  a contradiction  to that $\min\{\psi,\xi\}\leq \phi$.   \end{proof}
The following conclusion is an analogy, in the metric setting, of the fact that the specialization order of a sober topological space is directed complete.
\begin{prop}\label{Sober implies Yoneda} The specialization metric  of a sober approach space is Yoneda complete. \end{prop}
\begin{proof}
Let $(X,\delta)$ be a sober approach space and $d=\Omega(\delta)$ be its specialization metric.  Assume that $\{x_\lam\}$ is a forward Cauchy net in $(X,d)$. Then $\sup_\lam \delta(-,A_\lam)$ is an approach prime of $(X,\delta)$ by Lemma \ref{approach prime}. Since $(X,\delta)$ is sober, there exists $a\in X$ such that \[\sup_\lam \delta(-,A_\lam)=\delta(-,\{a\})=d(-,a).\] We claim that $a$ is a Yoneda limit of $\{x_\lam\}$, i.e.,  for all $x\in X$ \[\inf_\lam\sup_{\sigma\geq\lam}d(x_\sigma,x)=d(a,x).\]

For each $\lam$, since $\delta(a,A_\lam)\leq\sup_\sigma \delta(a,A_\sigma)=d(a,a)=0$,   then  \[d(a,x)=\delta(a,\{x\}) \leq \delta(a,A_\lam)+\sup_{\sigma\geq\lam}d(x_\sigma,x) =\sup_{\sigma\geq\lam}d(x_\sigma,x)  \] by (A4), hence
\[ d(a,x)\leq \inf_\lam\sup_{\sigma\geq\lam}d(x_\sigma,x).\]

For the converse inequality, we first show that $$\inf_\lam\sup_{\sigma\geq\lam}\sup_\tau\delta(x_\sigma, A_\tau)=0.$$ Given $\varepsilon>0$, since  $\{x_\lam\}$ is  forward Cauchy, there is some $\lam_0$ such that $d(x_\mu,x_\nu)<\varepsilon$ whenever $\nu\geq\mu\geq\lam_0$. Then for any index $\tau$ and any $\sigma\geq\lam_0$, \[\delta(x_\sigma,A_\tau)\leq\inf_{\mu\geq\sigma,\tau} d(x_\sigma,x_\mu)<\varepsilon,\] it follows that $\sup_{\sigma\geq\lam_0}\sup_\tau\delta(x_\sigma,A_\tau) \leq\varepsilon$, hence \[\inf_\lam\sup_{\sigma\geq\lam}\sup_\tau\delta(x_\sigma,A_\tau) =0\] by  arbitrariness of $\varepsilon$.   Therefore, \begin{align*}\inf_\lam\sup_{\sigma\geq\lam}d(x_\sigma,x)&\leq \inf_\lam\sup_{\sigma\geq\lam}(d(x_\sigma,a)+d(a,x))\\ &= \inf_\lam\sup_{\sigma\geq\lam}\sup_\tau\delta(x_\sigma,A_\tau) +d(a,x) & (d(x_\sigma,a)=\sup_\lam \delta(x_\sigma,A_\lam))\\ &=d(a,x).\end{align*}
  This completes the proof.
\end{proof}

\begin{lem} \label{metric approach prime} For each metric space $(X,d)$,  the  approach primes of $\Gamma(X, d)$  are  exactly the flat weights of $(X,d)$.
\end{lem}
\begin{proof}
Given an approach prime $\phi$ of $\Gamma(X, d)$, we show that $\phi$ is a flat weight of $(X,d)$. It suffices to check that $\phi$ satisfies the condition (b) in Proposition \ref{flat weight}.

Suppose $\phi(x_i)<\varepsilon_i\ (i=1,2)$. Consider the functions $\psi(x)=\max\{0, \varepsilon_1-d(x_1,x)\}$ and $\xi(x)=\max\{0, \varepsilon_2-d(x_2,x)\}$.  It is easy to check that $\psi$ and $\xi$ are regular functions satisfying  $\psi\nleqslant \phi$ and $\xi\nleqslant \phi$ ($\psi(x_1)=\varepsilon_1$, $\xi(x_2)=\varepsilon_2$). Since $\phi$ is an approach prime, we have $\min\{\psi,\xi\}\nleqslant \phi$. Thus, there exists $y\in X$ such that $\phi(y)<\min\{\psi(y),\xi(y)\}$, namely $$\phi(y)<\varepsilon_1-d(x_1,y)\ \ {\rm and}\ \ \phi(y)<\varepsilon_2-d(x_2,y).$$
So, there exists $\varepsilon>0$ such that $\phi(y)<\varepsilon$ and $d(x_i,y)+\varepsilon<\varepsilon_i\ (i=1,2)$.

Conversely, we show that each flat weight $\phi$ of $(X,d)$ is an approach prime of $(X,\Gamma(d))$. By Proposition \ref{flat weight},   $\phi=\inf_\lam\sup_{\sigma\geq\lam} d(-,x_\sigma)$ for a forward Cauchy net $\{x_\lam\}$ in $(X,d)$. For each $x\in X$, $\{d(x,x_\lam)\}$ is a forward Cauchy net in $([0,\infty],d_L)$, hence, thanks to Equation (\ref{forward cauchy nets converge}),
\[\phi(x)=\inf_\lam\sup_{\sigma\geq\lam} d(x,x_\sigma)= \sup_\lam\inf_{\sigma\geq\lam}d(x,x_\sigma)= \sup_\lam\Gamma(d)(x,A_\lam),\] where  $A_\lam=\{x_\sigma\mid \sigma\geq\lambda\}$. Therefore, $\phi$ is an approach prime of $(X,\Gamma(d))$, by  Lemma \ref{approach prime}.
\end{proof}

Now we come to the main result in this section.
\begin{thm}\label{5.8} For a metric space $(X,d)$,
the specialization metric space of the sobrification of $\Gamma(X, d)$ coincides with the Yoneda completion of $(X,d)$. \end{thm}

\begin{proof} Let $(\widehat{X},\widehat{\Gamma(d)})$ denote the sobrification of $\Gamma(X, d)$, and $(\CF X,\overline{d})$ the Yoneda completion of $(X,d)$.
 By Lemma \ref{metric approach prime},   $\widehat{X}$ and $\CF X$ have the same elements, i.e.,  the flat weights of $(X,d)$. For any flat weights $\phi,\psi$ of $(X,d)$, we have by Equation (\ref{widehat distance}) that \[\widehat{\Gamma(d)}(\phi,\{\psi\})=\widehat{\psi}(\phi) =\sup_{x\in X}d_L(\phi(x),\psi(x))= \overline{d}(\phi,\psi),\]  showing that the specialization metric   of $(\widehat{X},\widehat{\Gamma(d)})$ coincides with $ \overline{d}$.
\end{proof}

\begin{prop}\label{P is sober} The approach space $\mathbb{P}$ is the sobrification of the metric approach space $\Gamma([0,\infty), d_R)$. %, where $\delta_{d_R}$ is  the approach distance generated by the metric $d_R$ on $[0,\infty)$.
In particular, $\mathbb{P}$ is sober.\end{prop}

\begin{proof} Suppose that  $\phi$ is an approach prime of $([0,\infty), \Gamma(d_R))$. Then there is a forward Cauchy net $\{a_\lam\}$ in $([0,\infty), d_R)$ such that $\phi=\inf_\lam\sup_{\sigma\geq\lam} d_R(-,a_\sigma)$. If $\{a_\lam\}$ is eventually a Cauchy net of real numbers in the usual sense, then $\phi =d_R(-,a)$, where $a=\lim_\lam a_\lam$. If $\{a_\lam\}$ is an almost increasing net that tends to infinity, then $\phi$ is the constant function $\underline{0}$ on $([0,\infty), d_R)$ with value $0$.

Define a map \[f:\mathbb{P}\rightarrow (\widehat{[0,\infty)}, \widehat{\Gamma(d_R)})\] by  $f(a)= d_R(-,a)$ for all $a\in[0,\infty)$ and $f(\infty)=\underline{0}$.
We claim that $f$ is an isomorphism of approach spaces. Since $f$ is clearly  a bijection, we only need to check that $\delta_\mathbb{P}(b,A)=\widehat{\Gamma(d_R)}(f(b),f(A))$ for all $b\in[0,\infty]$ and (non-empty) $A\subseteq[0,\infty]$.

 By Equation (\ref{widehat distance*}),
\[\widehat{\Gamma(d_R)}(f(b),f(A))= \sup\{\widehat{\phi}(f(b))\mid \phi\in\mathcal{R}[0,\infty),\forall a\in A, \widehat{\phi}(f(a))=0\},\]
where, $\mathcal{R}[0,\infty)$ denotes the set of  regular functions of $([0,\infty), \Gamma(d_R))$. For each $\phi\in\mathcal{R}[0,\infty)$, we have $\widehat{\phi}(f(a))=\widehat{\phi}(d_R(-,a))=\phi(a)$ for $a\in[0,\infty)$ and  $\widehat{\phi}(f(\infty))= \widehat{\phi}(\underline{0})=\sup_{x\in[0,\infty)}\phi(x)$. We proceed with three cases.

\textbf{Case 1}.  $\sup A=\infty$. In this case, the constant function $\underline{0}$ is the only  regular function of $([0,\infty), \Gamma(d_R))$  that satisfies the condition that $\widehat{\phi}(f(a))=0$ for all $a\in A$,  hence  \[\delta_\mathbb{P}(b,A)=0 =\widehat{\underline{0}}(f(b))= \widehat{\Gamma(d_R)}(f(b),f(A)).\]

\textbf{Case 2}. $b=\infty$, $\sup A<\infty$. Since the regular function $\phi=d_R(-, \sup A)$ satisfies the condition that $\widehat{\phi}(f(a))=0$ for all $a\in A$,  \[\widehat{\Gamma(d_R)}(f(b),f(A)) \geq \widehat{\phi}(f(b)) = \widehat{\phi}(\underline{0}) = \sup_{x\in[0,\infty)}d_R(x,\sup A)=\infty,\] it follows that $\delta_\mathbb{P}(b,A) =\infty=\widehat{\Gamma(d_R)}(f(b),f(A))$.

\textbf{Case 3}. $b<\infty$, $\sup A<\infty$. Since $\psi= d_R(-,\sup A)$ is a regular function of  $([0,\infty), \Gamma(d_R))$ such that $\widehat{\psi}(f(a))=0$ for all $a\in A$, it follows that \[ \delta_\mathbb{P}(b,A) = d_R(b,\sup A) = \widehat{\psi}(f(b))\leq \widehat{\Gamma(d_R)}(f(b),f(A)). \] Conversely, let $\phi$ be a regular function   on  $([0,\infty), \Gamma(d_R))$ such that $\widehat{\phi}(f(a))=\phi(a)=0$ for all $a\in A$. Since $\phi$ is a weight of $([0,\infty), d_R)$ by Proposition \ref{regular function in MAS}, it follows that \[\widehat{\phi}(f(b))=\phi(b)\leq\phi(a)+d_R(b,a) =d_R(b,a) \] for all $a\in A$, hence $\widehat{\phi}(f(b))\leq\inf_{a\in A}d_R(b,a)=d_R(b,\sup A)$. Therefore \[\widehat{\Gamma(d_R)}(f(b),f(A))\leq d_R(b,\sup A)= \delta_\mathbb{P}(b,A),\] completing the proof.
\end{proof}

\section{Sober metric approach spaces}

In this section we characterize  metric approach spaces that are sober as exactly the Smyth complete spaces.

\begin{defn} \label{Smyth complete} A  metric space  is Smyth complete if it is  separated  and all of its forward Cauchy nets  converge  in its symmetrization.  \end{defn}

The metric space $([0,\infty],d_L)$ is Smyth complete. But, $([0,\infty],d_R)$ is not, though it is Yoneda complete.

Smyth completeness originated in the works of Smyth \cite{Smyth87,Smyth94}. The above postulation  is taken from \cite{Goubault,KS2002}.  For more information on  Smyth completeness the reader is referred  to \cite{BvBR1998,FS02,Goubault,KSF,KS2002}. In these works, Smyth completeness is more or less related to the topological properties of the spaces under consideration. However, as shown below,  if we view metric spaces as categories enriched over Lawvere's quantale $([0,\infty]^{\rm op},+)$,  Smyth completeness for metric spaces  can be formulated purely in categorical terms: a metric space is Smyth complete if it is separated and all of its flat weights are representable. This shows, in close resemblance to Lawvere's postulation of complete metric spaces (i.e., every Cauchy weight is representable), that Smyth completeness is a categorical property.  This  can be thought of as an example for   ``whether Lawvere's work has any bearing on what we are doing here", a question raised  by Smyth in \cite{Smyth87}.

We need some preparations. A net $\{x_\lam\}$ in a metric space $(X,d)$ is biCauchy \cite{KS2002} if \[\inf_\lam\sup_{\nu,\mu\geq\lam}d(x_\mu,x_\nu)=0.\] %That is, for any $\varepsilon>0$, there exists $\lambda_0$ such that $d(x_{\mu},x_{\nu}) <\varepsilon$ whenever $\nu,\mu\geq \lambda_0$.

Every forward Cauchy net in $([0,\infty],d_L)$ is biCauchy. The sequence $\{n\}$   in $([0,\infty],d_R)$ is forward Cauchy, but not biCauchy.

\begin{lem}\label{bicauchy=cauchy} A forward Cauchy net  $\{x_\lam\}$  in a metric space $(X,d)$ is biCauchy if and only if the weight $\phi=\inf_\lam\sup_{\sigma\geq\lam} d(-,x_\sigma)$ is   Cauchy. \end{lem}
\begin{proof}
If $\{x_\lam\}$ is biCauchy, it is easily verified that the coweight $\psi=\inf_\lam\sup_{\sigma\geq\lam} d(x_\sigma,-)$ is a left adjoint of $\phi$, hence $\phi$ is Cauchy.
Conversely, suppose that $\{x_\lam\}$  is a forward Cauchy net and $\phi=\inf_\lam\sup_{\sigma\geq\lam} d(-,x_\sigma)$ is a Cauchy weight. By Equation (\ref{left adjoint}) the left adjoint $\psi$ of $\phi$ is given by \[\psi(x) =\overline{d}(\phi,d(-,x)).\] Since  $\phi$ is a Yoneda limit of the forward Cauchy net $\{d(-,x_\lam)\}$ in $(\mathcal{F}X, \overline{d})$ by Theorem \ref{Yoneda complete}, it follows that for all $x\in X$, \begin{align*}
  \psi(x) &=\overline{d}(\phi,d(-,x)) \\
  &=\inf_\lam\sup_{\sigma\geq\lam}\overline{d}(d(-,x_\sigma),d(-,x))\\
  &=\inf_\lam\sup_{\sigma\geq\lam}d(x_\sigma,x).
\end{align*}
Therefore,
\[\inf_\lam\sup_{\sigma,\mu\geq\lam}d(x_\sigma,x_\mu)\leq \inf_{x\in X}\inf_\lam\sup_{\sigma,\mu\geq\lam}(d(x,x_\mu)+d(x_\sigma,x)) = \inf_{x\in X}(\phi(x)+\psi(x))=0,\] showing that $\{x_\lam\}$  is biCauchy. \end{proof}

The above lemma is similar to Proposition 4.13 in Hofmann and Reis \cite{HR2013}. However, there is a subtle difference. Proposition 4.13 in   \cite{HR2013} says that for every net $\{x_\lam\}$ in a metric space, the coweight $\psi=\inf_\lam\sup_{\sigma\geq\lam} d(x_\sigma,-)$ is   left adjoint to the weight  $\phi=\inf_\lam\sup_{\sigma\geq\lam} d(-,x_\sigma)$ if and only if $\{x_\lam\}$ is biCauchy. The above lemma shows that for a forward Cauchy net  $\{x_\lam\}$, if the weight  $\phi=\inf_\lam\sup_{\sigma\geq\lam} d(-,x_\sigma)$ has a left adjoint, then this left adjoint must be $\psi=\inf_\lam\sup_{\sigma\geq\lam} d(x_\sigma,-)$ and $\{x_\lam\}$ is biCauchy.

\begin{prop} A metric space $(X,d)$ is Smyth complete if and only if for each flat weight $\phi$ of $(X,d)$, there is a unique $a\in X$ such that $\phi=d(-,a)$. \end{prop}

\begin{proof}
\textbf{Sufficiency}. That $(X,d)$ is separated is obvious. Given a forward Cauchy net $\{x_\lam\}$ in $(X,d)$, let $\phi=\inf_\lam\sup_{\sigma\geq\lam} d(-,x_\sigma)$. Then $\phi$ is a flat weight, hence $\phi=d(-,a)$ for some $a\in X$. We leave it to the reader to check  that $\{x_\lam\}$ converges to $a$  in $(X,d^{\rm sym})$.

 \textbf{Necessity}. Let $\phi$ be a flat weight of $(X,d)$. By Proposition \ref{flat weight}, there is a forward Cauchy net $\{x_\lam\}$ in $(X,d)$ such that $\phi=\inf_\lam\sup_{\sigma\geq\lam} d(-,x_\sigma)$. By assumption, $\{x_\lam\}$ has a unique limit, say $a$, in $(X,d^{\rm sym})$. So, $\{x_\lam\}$ is a biCauchy net in $(X,d)$ with $a$ as a Yoneda limit. Thus,  $\phi$ is a Cauchy weight by Lemma \ref{bicauchy=cauchy} and $\colim\phi=a$ by Proposition \ref{supremum}. Then,  by Equation (\ref{left adjoint}) and Equation (\ref{colim}), $d(a,-)$ is a left adjoint of $\phi$, hence $\phi=d(-,a)$.
  \end{proof}

If there is an isometric map from a metric space $(X,d)$ to a Smyth complete metric space $(Y,p)$, then, by Proposition \ref{image of Cauchy weight}, every flat weight of $(X,d)$ will be a Cauchy weight. This leads to the following

\begin{defn}A metric space   is Smyth completable if  all of its flat weights   are Cauchy. \end{defn}

 The following conclusion says that the above postulation of Smyth completable metric spaces is equivalent to that in \cite{KS2002}.

\begin{prop} A metric space   is  Smyth completable  if and only if all of its forward Cauchy nets are biCauchy. \end{prop}

\begin{proof} \textbf{Sufficiency}. Let $\phi$ be a flat weight. By Proposition \ref{flat weight}, there is a forward Cauchy net  $\{x_\lam\}$ such that $\phi=\inf_\lam\sup_{\sigma\geq\lam} d(-,x_\sigma)$. By assumption, $\{x_\lam\}$ is a biCauchy net, thus, $\phi$ is a Cauchy weight by Lemma \ref{bicauchy=cauchy}.

\textbf{Necessity}. Let $\{x_\lam\}$  be a forward Cauchy net.  Then $\phi=\inf_\lam\sup_{\sigma\geq\lam} d(-,x_\sigma)$ is a flat weight by Proposition \ref{flat weight}, hence a Cauchy weight by assumption.  Thus,  $\{x_\lam\}$  is   biCauchy  by Lemma \ref{bicauchy=cauchy}. \end{proof}

\begin{thm} Let $(X,d)$ be a metric space.
The following are equivalent: \begin{enumerate}[(1)] \item $(X,d)$ is Smyth completable.
\item The sobrification of $\Gamma(X, d)$ is a metric approach space.
\item The Yoneda completion on $(X,d)$ is idempotent, i.e., the map $\y_{\CF X}: (\CF X,\overline{d})\rightarrow (\CF(\CF X),\overline{\overline{d}})$ is surjective.
\end{enumerate}  In this case, the sobrification of $\Gamma(X, d)$ is generated by the Cauchy completion of  $(X,d)$. \end{thm}

\begin{proof} $(1)\Rightarrow(2)$ First of all, by virtue of Lemma \ref{metric approach prime}, each approach prime of $(X,\Gamma(d))$ is a flat weight of $(X,d)$, hence a Cauchy weight of $(X,d)$.

If we could show that for each approach prime $\psi$ of $(X,\Gamma(d))$ and every non-empty set $\{\phi_i\}_{i\in I}$ of approach primes of $(X,\Gamma(d))$, it holds that \[\widehat{\Gamma(d)}(\psi,\{\phi_i\}_{i\in I})=\inf_{i\in I}\overline{d}(\psi,\phi_i),\] then the sobrification of $(X,\Gamma(d))$ will be a metric approach space, generated by  $(\CF X,\overline{d})$, the Yoneda completion of $(X,d)$. To see this, we calculate:
 \begin{align*}
  \widehat{\Gamma(d)}(\psi,\{\phi_i\}_{i\in I}) &=\sup\{\widehat{\xi}(\psi)\mid \xi\in\mathcal{R}X,~\forall i\in I,~\widehat{\xi}(\phi_i)=0\} \\
&=\sup\{\widehat{\xi}(\psi)\mid \xi\in\mathcal{R}X,~\forall i\in I,~\xi\leq \phi_i\}\\
&=\sup\{\widehat{\xi}(\psi)\mid \xi\in\mathcal{R}X,~\xi\leq \inf_{i\in I}\phi_i\}\\
&=\widehat{\inf_{i\in I}\phi_i}(\psi)&(\inf_{i\in I}\phi_i\in\mathcal{R}X )\\
&=\sup_{x\in X}d_L(\psi(x),\inf_{i\in I}\phi_i(x))\\
&=\inf_{i\in I}\sup_{x\in X}d_L(\psi(x),\phi_i(x))&({\rm Lemma~ \ref{Cauchy weight}})\\
&=\inf_{i\in I}\overline{d}(\psi,\phi_i),
\end{align*} where, $\mathcal{R}X$ denotes the set of regular functions of $(X,\Gamma(d))$.

$(2)\Rightarrow(3)$
Since the sobrification $(\widehat{X},\widehat{\Gamma(d)})$ of $(X,\Gamma(d))$ is a metric approach space, it must be generated by the Yoneda completion $(\CF X,\overline{d})$ of $(X,d)$ by Theorem \ref{5.8}.  Given a flat weight $\Phi:\CF X\rightarrow[0,\infty]$ of $(\CF X,\overline{d})$, it follows from Lemma \ref{metric approach prime} that $\Phi$ is an approach prime of $(\widehat{X},\widehat{\Gamma(d)})$. Thus, there exists a unique $\xi\in \widehat{X}(=\CF X)$ such that \[\Phi=\widehat{\Gamma(d)}(-,\{\xi\}) = \overline{d}(-,\xi).\] This shows that  $\y_{\CF X}: (\CF X,\overline{d})\rightarrow (\CF(\CF X),\overline{\overline{d}})$ is surjective,  the conclusion thus follows.

$(3)\Rightarrow(1)$
 If $\phi$ is flat, then $\y_X(\phi)$ is a flat weight of $(\CF X,\overline{d})$. Thus, $\y_X(\phi)=\overline{d}(-,\psi)$ for some $\psi\in \CF X$ since $\y_{\CF X}: (\CF X,\overline{d})\rightarrow (\CF(\CF X),\overline{\overline{d}})$ is surjective. This shows that $\y_X(\phi)$ is a Cauchy weight of $(\CF X,\overline{d})$. Then, applying Proposition \ref{image of Cauchy weight}(3) to $\y_X$ gives that $\phi$ is Cauchy.

In this case,  the Cauchy completion and the Yoneda completion coincide with each other.  Hence, the final claim follows from Theorem \ref{5.8}. \end{proof}

\begin{cor}(\cite{BRC}) Let $(X,d)$ be a symmetric   metric space. Then the sobrification of $(X,\Gamma(d))$ is a metric approach space and is generated by the Cauchy completion of $(X,d)$. \end{cor}

\begin{proof}This follows from   that every symmetric metric space is Smyth completable. \end{proof}

\begin{thm}\label{main} Let $(X,d)$ be a metric space. The following are equivalent: \begin{enumerate}[(1)] \item The approach space $(X,\Gamma(d))$ is sober.
\item $(X,d)$ is Smyth complete.
\item $(X,d)$ is a fixed point of the Yoneda completion, i.e., $\y_X:(X,d)\rightarrow (\CF X,\overline{d})$ is an isomorphism. \end{enumerate} \end{thm}
\begin{proof} $(1)\Rightarrow(2)$ If $\phi$ is flat weight of $(X,d)$, then $\phi$  is an approach prime of $(X,\Gamma(d))$ by Lemma \ref{metric approach prime}, so, there is a unique $a\in X$ such that $\phi=\Gamma(d)(-,\{a\})=d(-,a)$, showing that $(X,d)$ is Smyth complete.

$(2)\Rightarrow(3)$ This follows from the construction of the Yoneda completion and the fact that all flat weights of $(X,d)$ are of the form $d(-,a)$.

$(3)\Rightarrow(1)$  Let $\phi$ be an approach prime of $(X,\Gamma(d))$. By Lemma \ref{metric approach prime}, $\phi$ is a flat weight of $(X,d)$, hence an element of the Yoneda completion of $(X,d)$. Since $(X,d)$ is a fixed point of the Yoneda completion, there is a unique $a\in X$ such that $\phi=d(-,a) =\Gamma(d)(-,\{a\})$. Hence  $(X,\Gamma(d))$ is sober.
\end{proof}

\noindent{\bf Acknowledgement} The authors thank cordially the referee for her/his most valuable comments and  helpful suggestions.

\end{document}